\documentclass{article}
\usepackage{comment}
\usepackage{amsthm}
\usepackage{mathrsfs}
\usepackage{fullpage}
\usepackage{setspace}
\usepackage{url}
\usepackage{tikz}
\usepackage{mathtools}
\usetikzlibrary{er}
\usetikzlibrary{arrows.meta}
\usepackage{amsmath,amsfonts,amssymb}
\usepackage{verbatim,enumitem,graphics}
\usepackage{microtype}
\usepackage{xcolor}
\usepackage[backref=page]{hyperref}
\hypersetup{%
    colorlinks,
    linkcolor={red!50!black},
    citecolor={green!50!black},
    urlcolor={blue!50!black}
}
\usepackage{cleveref}
\usepackage{dsfont}
\usepackage[abbrev,msc-links,backrefs]{amsrefs} 
\usepackage{doi}
\usepackage{lineno}
\usepackage[normalem]{ulem}
\usepackage{mathdots}

\newtheorem{theorem}{Theorem}[section]
\newtheorem{lemma}[theorem]{Lemma}
\newtheorem{proposition}[theorem]{Proposition}

\newtheorem{conjecture}[theorem]{Conjecture}

\theoremstyle{remark}
\newtheorem*{remark}{Remark}

\newcommand{\exaff}{\mathrm{ex}_{\mathrm{aff}}}
\newcommand{\omegaaff}{\omega_{\mathrm{aff}}}
\newcommand{\rankaff}{\rank_{\mathrm{aff}}}
\newcommand{\Raff}[1]{R_{\mathrm{aff } #1}}
\newcommand{\homaff}{\mathrm{hom}_{\mathrm{aff}}}
\newcommand{\autaff}{\mathrm{Aut}_{\mathrm{aff}}}
\newcommand{\sqbinom}[2]{\begin{bsmallmatrix} #1 \\ #2 \end{bsmallmatrix}}

\newcommand{\rank}{\mathrm{rank}}

\newcommand{\spn}{\mathrm{span}}

\let\epsilon=\varepsilon

\begin{document}

\setstretch{1.27}

\title{Vector space Ramsey numbers and weakly Sidorenko affine configurations}

\author{Bryce~Frederickson \thanks{Department of Mathematics, Emory University, 
Atlanta, Georgia 30322. Email: {\tt bfrede4@emory.edu}.}
\and
Liana~Yepremyan \thanks{Department of Mathematics, Emory University, 
Atlanta, Georgia 30322. Email: {\tt lyeprem@emory.edu}. The research of the second author is  supported by NSF grant 2247013: Forbidden and Colored Subgraphs.}}

\date{}

\maketitle

\begin{abstract}
For $B \subseteq \mathbb F_q^m$, the $n$-th affine extremal number of $B$ is the maximum cardinality of a set $A \subseteq \mathbb F_q^n$ with no subset which is affinely isomorphic to $B$. Furstenberg and Katznelson proved that for any $B \subseteq \mathbb F_q^m$, the $n$-th affine extremal number of $B$ is $o(q^n)$ as $n \to \infty$. By counting affine homomorphisms between subsets of $\mathbb F_q^n$, we derive new bounds and give new proofs of some previously known bounds for certain affine extremal numbers. At the same time, we establish corresponding supersaturation results. We connect these bounds to certain Ramsey-type numbers in vector spaces over finite fields. For $s,t \geq 1$, let $R_q(s,t)$ denote the minimum $n$ such that in every red-blue coloring of the one-dimensional subspaces of $\mathbb F_q^n$, there is either a red $s$-dimensional subspace or a blue $t$-dimensional subspace of $\mathbb F_q^n$. The existence of these numbers is a special case of a well-known theorem of Graham, Leeb, Rothschild. We improve the best known upper bounds on $R_2(2,t)$, $R_3(2,t)$, $R_2(t,t)$, and $R_3(t,t)$.
\end{abstract}

\section{Introduction}
We consider bounds for Ramsey-type and Tur\'an-type problems in the setting of vector spaces over finite fields. In this paper, we use $\sqbinom{V}{t}$ to denote the collection of all $t$-dimensional linear subspaces of a vector space $V$. The following theorem is a special case of a classical theorem of Graham, Leeb, and Rothschild \cite{GLR72}, which establishes the existence of the Ramsey numbers we consider.

\begin{theorem}[Graham, Leeb, Rothschild]\label{thm:GLR72}
Let $\mathbb F_q$ be any finite field. For any positive integers $t_1, \ldots, t_k$, there exists a minimum $n =: R_q(t_1, \ldots, t_k)$ such that for every $k$-coloring $f : \sqbinom{\mathbb F_q^n}{1} \to [k]$ of the $1$-dimensional linear subspaces of $\mathbb F_q^n$, there exist $i \in [k]$ and a linear subspace $U \subseteq \mathbb F_q^n$ of dimension $t_i$, such that $\sqbinom{U}{1}$ is monochromatic in color $i$.
\end{theorem}

In the case $t_1= \cdots = t_k = t$, we write $R_q(t_1, \ldots, t_k) = R_q(t ;k)$. The bounds for $R_q(t_1, \ldots, t_k)$ implied by early proofs of Theorem \ref{thm:GLR72} (see \cite{GLR72}, \cite{Sp79}) are quite large due to repeated use of the Hales-Jewett Theorem \cite{HJ63}. In the case $q=2$, the problem can be reduced to the disjoint unions problem for finite sets, considered by Taylor \cite{T81}, which gives the following bound. 
\begin{theorem}[Taylor]\label{thm:T81}
The number $R_2(t;k)$ is at most a tower of height $2k(t-1)$ of the form
\[R_2(t;k) \leq k^{3^{k^{\iddots^3}}}.\]
\end{theorem}

For comparison, lower bounds for $R_2(t;k)$ attained from applying the techniques from \cite{A16} such as the Lov\'asz Local Lemma to a uniform random coloring are only on the order of $\Omega\left(\frac{2^t}{t}\log_2 k\right).$

We improve the bound of Theorem \ref{thm:T81} by bringing the height of the tower down to $(k-1)(t-1) + o(t)$, and we prove a corresponding bound over $\mathbb F_3$. 

\begin{theorem}\label{thm:diagonal Fq k colors}
    There exists a constant $C_0 \approx 13.901$ such that for $\sigma_2 = 2$ and $\sigma_3 = C_0$, the following holds for $q \in \{2,3\}$. For any $k \geq 2$ and any $t_k \geq \cdots \geq t_1 \geq 2$, $R_q(t_1, \ldots, t_k)$ is at most a tower of height $\sum_{i = 1}^{k-1}(t_i-1) + 1$ of the form
    \[R_q(t_1, \ldots, t_k) \leq \sigma_q^{\sigma_q^{\iddots^{\sigma_q^{3t_k}}}}.\]
\end{theorem}

More recently, Nelson and Nomoto \cite{NN21} considered the off-diagonal version of this problem over $\mathbb F_2$ with two colors while investigating $\chi$-boundedness of certain classes of binary matroids, and they proved the following bound.
\begin{theorem}[Nelson, Nomoto] \label{thm:Nelson Nomoto}
For every $t \geq 2$,
\[R_2(2,t) \leq (t+1)2^t.\]
\end{theorem}
In this case, standard probabilistic arguments give a lower bound for $R_2(2,t)$ which is only linear in $t$. Nelson and Nomoto asked if a subexponential upper bound is possible. While the answer to that question remains to be seen, we provide the following exponential improvement for $R_2(2,t)$ and similarly give the first exponential upper bound for $R_3(2,t)$.
\begin{theorem} \label{thm:off-diagonal Fq}
There exists a constant $C_0 \approx 13.901$ such that as $t \to \infty$,
\begin{align*}
    R_2(2,t) &= O\big(t 6^{t/4}\big); \\
    R_3(2,t) &= O\big(t C_0^t\big).
\end{align*}
\end{theorem}

These improved bounds come from some simple observations about affine extremal numbers and their supersaturation properties, which are analogous to bipartite Tur\'an numbers in graph theory. The $n$-th \emph{affine extremal number} of a family $\mathcal B$ of affine configurations $\{B_i \subseteq \mathbb F_q^{m_i}\}_{i \in I}$, denoted $\exaff(n,\mathcal B)$, is the maximum size of a subset $A \subseteq \mathbb F_q^n$ with no affine copy of any $B_i$ (see Section \ref{sec:preliminaries} for a more complete definition). The asymptotic study of affine extremal numbers dates back at least to the following theorem of Furstenberg and Katznelson \cite{FK85}.

\begin{theorem}[Furstenberg, Katznelson]\label{thm:Furstenberg Katznelson}
    Let $\mathbb F_q$ be any finite field. For any $t \geq 0$,
    \[\exaff(n,\mathbb F_q^t) = o(q^n) \qquad \text{as} \qquad n \to \infty.\]
\end{theorem}
Since any $\mathcal B$-free set is $\mathbb F_q^t$-free for some $t$, \Cref{thm:Furstenberg Katznelson} says that affine extremal numbers are always $o(q^n)$. Furstenberg
and Katznelson went on to prove a density version of the Hales-Jewett Theorem \cite{FK91}, from which \Cref{thm:Furstenberg Katznelson} is immediate. Alternative proofs of these results can be found in \cite{RTST06} and \cite{polymath12}, respectively.

The projective version of this problem is even older, beginning with the following result of Bose and Burton \cite{BB66}.
\begin{theorem}[Bose, Burton]\label{thm:Bose Burton}
Let $\mathbb F_q$ be a finite field, and let $t \geq 1$. Let $\mathcal A$ be a subset of $\sqbinom{\mathbb F_q^n}{1}$ for which there is no linear $t$-dimensional subspace $U \subseteq \mathbb F_q^n$ with $\sqbinom{U}{1} \subseteq \mathcal A$. Then
\[|\mathcal A| \leq \frac{q^n - q^{n-t+1}}{q-1},\]
with equality if and only if $\mathcal A = \sqbinom{\mathbb F_q^n}{1} \setminus \sqbinom{W}{1}$ for some linear $(n-t+1)$-dimensional linear subspace $W \subseteq \mathbb F_q^n$.
\end{theorem}

\begin{remark}
    It is sometimes convenient to identify a set $\mathcal A \subseteq \sqbinom{\mathbb F_q^n}{1}$ of projective points with a set $A \subseteq \mathbb F_q^n \setminus \{0\}$ of vectors, given by
\[A = \bigcup_{\ell \in \mathcal A} \ell \setminus \{0\}.\]
We call a set $A$ of this form \emph{projectively determined}. Similarly, we can identify any $k$-coloring of $\sqbinom{\mathbb F_q^n}{1}$ with a 
\emph{projectively determined} $k$-coloring of $\mathbb F_q^n \setminus \{0\}$, meaning that each color class is a projectively determined set of vectors. Moving forward, we will work from the perspective of projectively determined sets and colorings whenever we discuss results of a projective nature, such as Theorems \ref{thm:diagonal Fq k colors}, \ref{thm:off-diagonal Fq}, and \ref{thm:Bose Burton}.
\end{remark}

The problem of determining projective extremal numbers asymptotically for general projective configurations over $\mathbb F_q$ was almost entirely solved by Geelen and Nelson \cite{GN15}, who proved a theorem analogous to the Erd\H{o}s-Stone-Simonivits Theorem for graphs. Their theorem gives precise asymptotics for the extremal number of any projective configuration, except for those which exclude a linear hyperplane, which are usually called ``affine''. Up to a constant factor, the projective extremal numbers of these ``affine'' projective configurations reduce to affine extremal numbers of the type discussed in this paper. Therefore, what can be said of their projective extremal numbers is that they are degenerate by \Cref{thm:Furstenberg Katznelson}, which is far from an asymptotic determination, similar to the case of extremal numbers of bipartite graphs.

It is unknown in general (see \cite{G21}, Open Problem 32) whether the $o(q^n)$ bound in \Cref{thm:Furstenberg Katznelson} can be taken to be of the form $O\left((q^{1-\delta})^n\right)$ for some $\delta = \delta(q,t) > 0$. However, for $q = 2$ and $q=3$, we have the following respective results of Bonin and Qin \cite{BQ00}, and of Fox and Pham \cite{FP17}.
\begin{theorem}[Bonin, Qin] \label{thm:Bonin Qin}
    There exists an absolute constant $c$ such that for every $t \geq 1$, every subset of $\mathbb F_2^n$ of size at least $(2^{1 - c2^{-t}})^n$ contains an affine $t$-space. 
\end{theorem}

\begin{theorem}[Fox, Pham]\label{thm:Fox Pham}
There exist absolute constants $c$ and $C_0$, with $C_0 \approx 13.901$ such that for every $t \geq 1$, every subset of $\mathbb F_3^n$ of size at least $\left(3^{1 - cC_0^{-t}}\right)^n$ contains an affine $t$-space.
\end{theorem}

The proof of Theorem \ref{thm:Bonin Qin} is entirely self-contained and is no more than a page. Theorem \ref{thm:Fox Pham}, on the other hand, is the culmination of several breakthroughs related to the Cap Set Problem, starting with the advances in polynomial methods from Croot, Lev, and Pach \cite{CLP17} and the subsequent proof of the Cap Set Theorem by Ellenberg and Gijswijt \cite{EG17}, which says that $\exaff(n,\mathbb F_3^1) \leq (3^{1 - \delta})^n$ for some $\delta > 0$. Fox and Lov\'asz \cite{FL17} then used this result to give improved bounds on Green's Arithmetic Triangle Removal Lemma \cite{G05}. Fox and Pham observed that this improvement implies a supersaturation version of the Cap Set Theorem, from which they derived Theorem \ref{thm:Fox Pham}, which is a multidimensional extension of the Cap Set Theorem. It is unknown whether the constant $C_0$ given in the theorem is tight, as probabilistic lower bounds for $\exaff(n,\mathbb F_3^t)$ are on the order of $\left(3^{1 - 3^{-(1 + o(1))t}}\right)^n$ \cite{FP17}.

The argument of Fox and Pham over $\mathbb F_3$ is essentially the same as Bonin and Qin's proof over $\mathbb F_2$. The key ingredient to both is supersaturation of affine lines, which is trivial over $\mathbb F_2$ and highly non-trivial over $\mathbb F_3$. We include this argument here in a more general form in Section \ref{sec:recovering bq and fp}. We also give a new proof of these results which additionally asserts a strong form of supersaturation, giving a quantitative improvement to a supersaturation result of Gijswijt \cite{G21} for certain affine configurations.

Our supersaturation results arise naturally from counting affine homomorphisms which are maps preserving affine configurations. (See \Cref{sec:preliminaries} for details of the notation and terminology used here). We use $\homaff(B,A)$ to denote the number of affine homomorphisms $B \to A$. We say that an affine configuration $B \subseteq \mathbb F_q^m$ is \emph{$C$-weakly Sidorenko} if, for any $A \subseteq \mathbb F_q^n$ of density $\alpha$,
\[\homaff(B, A) \geq \alpha^C N^{\rankaff(B)},\]
where $N := q^n$. By taking $A$ to be a $p$-random subset of $\mathbb F_q^n$ for some fixed $p \in (0,1)$, we see that $B$ cannot be $C$-weakly Sidorenko for $C < |B|$. In the case that $B$ is $C$-weakly Sidorenko with $C = |B|$, we simply say that $B$ is \emph{Sidorenko}; that is, $B$ is Sidorenko if
\[\homaff(B, A) \geq \alpha^{|B|}N^{\rankaff(B)}\]
for any $A \subseteq \mathbb F_q^n$ of density $\alpha$, with $N:= q^n$.

The notion of Sidorenko affine configurations originates from Saad and Wolf \cite{SW17}, who gave an equivalent definition in the language of linear forms. They proved that an affine configuration $\{x_1, \ldots, x_k\}$ with a single relation $\sum_{i = 1}^k \lambda_i x_i = 0$ is Sidorenko whenever the coefficients $\lambda_i$ can be partitioned into zero-sum pairs. They conjectured that these are the only affine configurations with a single relation which are Sidorenko. Fox, Pham, and Zhao showed that the conjecture is true in spirit, but in reality the correct statement is that $\{x_1, \ldots, x_k\}$ with a single relation $\sum_{i = 1}^k \lambda_i x_i = 0$ is Sidorenko if and only if the \emph{nonzero} coefficients $\lambda_i$ can be partitioned into zero-sum pairs \cite{FPZ21}. In particular, these results tell us that for each $k \geq 2$, the \emph{circuit} of length $2k$ over $\mathbb F_2$, which we define to be the affine configuration $C_{2k} := \{0, e_1, \ldots, e_{2k-2}, \sum_{i = 1}^{2k-2} e_i\} \subseteq \mathbb F_2^{2k-2}$, where $e_1, \ldots, e_{2k-2}$ are the standard basis vectors in $\mathbb F_2^{2k-2}$, is Sidorenko. We also have trivially that $\mathbb F_2^1$ is Sidorenko, as is any affinely independent affine configuration over any finite field. It is unknown whether there exist affine configurations over $\mathbb F_2$ which are not Sidorenko. Over $\mathbb F_3$, Fox and Pham \cite{FP17} pointed out that a result of Fox and Lov\'asz \cite{FL17} implies that for $C_0 \approx 13.901$, for any affine configuration $A \subseteq \mathbb F_3^n$ of size $\alpha N$, where $N := 3^n$, there are at least $\alpha^{C_0}N^2$ triples in $A^3$ of the form $(x,x+d,x+2d)$. In other words, $\mathbb F_3^1$ is $C_0$-weakly Sidorenko.
It is unclear whether or not $\mathbb F_3^1$ is $C$-weakly Sidorenko for some $C < C_0$. What we can say is that known lower bounds on $\exaff(n,\mathbb F_3^1)$ by Tyrell \cite{T22} imply that $\mathbb F_3^1$ is not $C$-weakly Sidorenko for any $C < 4.63$. 
We summarize these for future reference in the following lemma.
\begin{lemma}\label{lem:affine lines Sidorenko}
Let $\sigma_2 := 2$ and $\sigma_3 := C_0 \approx 13.901$ as in \Cref{thm:Fox Pham}.
\begin{enumerate}[label=(\alph*)]
    \item $\mathbb F_2^1$ is $\sigma_2$-weakly Sidorenko and hence Sidorenko.
    \item For each $k \geq 2$, $C_{2k} \subseteq \mathbb F_2^{2k-2}$ is Sidorenko.
    \item $\mathbb F_3^1$ is $\sigma_3$-weakly Sidorenko.
\end{enumerate}
\end{lemma}

We make the following simple observation which allows us to construct new weakly Sidorenko affine configurations from old ones, which we prove in \Cref{sec:homomorphic supersaturation}.
\begin{theorem}\label{thm:Sidorenko product}
    Suppose that $B_1 \subseteq \mathbb F_q^{m_1}$ is $C_1$-weakly Sidorenko and $B_2 \subseteq \mathbb F_q^{m_2}$ is $C_2$-weakly Sidorenko. Then $B_1 \times B_2$ is $C_1C_2$-weakly Sidorenko. In particular, if $B_1$ and $B_2$ are both Sidorenko, then so is $B_1 \times B_2$.
\end{theorem}

Recently, some attention has been given to classifying Sidorenko affine configurations, usually motivated by Ramsey multiplicity problems in additive settings. In addition to the work of \cite{SW17} and \cite{FPZ21} already mentioned, Kam\v{c}ev, Liebenau, and Morrison proved that affine configurations admitting a certain type of tree-like structure are Sidorenko \cite{KLM23}. They additionally gave a necessary condition for an affine configuration to be Sidorenko, which always holds trivially over $\mathbb F_2$. Altman defined a local weakening of the Sidorenko property and described a particular family of affine configurations (none over $\mathbb F_2$) which are not Sidorenko \cite{A22a}. See \cite{A22b}, \cite{CCS07}, \cite{KLM22}, \cite{KLP22}, \cite{V21}, and \cite{V23} for further work in the area.

We now give a vector space analogue of Sidorenko's Conjecture \cite{S91} on graph homomorphisms, which says that for any bipartite graph $H$ on $v$ vertices with $e$ edges, and any graph $G$ with $N$ vertices and $\alpha N^2/2$ edges, the number of homomorphisms from $H$ to $G$ is at least $\alpha^e N^v$.
\begin{conjecture}
    Every affine configuration over $\mathbb F_2$ is Sidorenko.
\end{conjecture}

It is not hard to check that $\mathbb F_q^1$ is not Sidorenko for any $q > 2$, so the conjecture only makes sense over $\mathbb F_2$.

\section{Preliminaries}\label{sec:preliminaries}

The objects we consider are subsets of finite-dimensional vector spaces over a fixed finite field $\mathbb F_q$. Such a subset $A \subseteq \mathbb F_q^n$ has linear structure as well as affine structure, which we define precisely below. Depending on which type of structure we are considering, we call $A$ a \emph{linear configuration} or an \emph{affine configuration}.

A \emph{linear relation} on $A = \{x_1, \ldots, x_k\} \subseteq \mathbb F_q^n$ is an equation of the form $\sum_{i = 1}^k \lambda_i x_i = 0,$ where $\lambda_1, \ldots, \lambda_k \in \mathbb F_q$.
If, in addition, we have $\sum_{i =1}^k \lambda_i = 0$, then the relation is called \emph{affine}. The relation is \emph{trivial} if each $\lambda_i = 0$. If $A$ has no nontrivial affine relations, then $A$ is called \emph{affinely independent}. A maximal affinely independent subset of $A$ is called an \emph{affine basis} for $A$. The size of any affine basis for $A$ is an invariant of $A$, called its \emph{affine rank}, which we denote by $\rankaff(A)$.

Given two configurations $A \subseteq \mathbb F_q^n$ and $B \subseteq \mathbb F_q^m$, a function $\varphi : B \to A$ is an \emph{affine homomorphism} if $\varphi$ preserves affine relations; that is, for any $\lambda_1, \ldots, \lambda_k \in \mathbb F_q$ and $x_1, \ldots, x_k \in B$ with $\sum_{i = 1}^k \lambda_i = 0$ and $\sum_{i =1}^k \lambda_i x_i = 0$, we have $\sum_{i = 1}^k \lambda_i \varphi\left(x_i\right) = 0.$ Equivalently, $\varphi$ is an affine homomorphism if $\varphi$ extends to an affine  map $\mathbb F_q^m \to \mathbb F_q^n$. We say that a homomorphism $\varphi : B \to A$ is an \emph{isomorphism} if $\varphi$ is bijective and $\varphi^{-1}$ is a homomorphism. If $B=A$, we call $\varphi$ an \emph{automorphism} of $B$, the set of which we denote by $\autaff(B)$. The affine isomorphism class of $A \subseteq \mathbb F_q^n$ is called its \emph{affine structure}. A homomorphism $\varphi$ which is an isomorphism onto its image is called \emph{non-degenerate}, which means that $\varphi$ is injective and preserves relations as well as non-relations.

Each of the affine notions above has a naturally-defined linear counterpart by considering linear relations instead of affine relations. In particular, the \emph{linear structure} of $A \subseteq \mathbb F_q^n$ is its linear isomorphism class, which characterizes the linear relations and non-relations among elements of $A$. We denote the linear rank of $A$ by $\rank(A)$.

We say that $A \subseteq \mathbb F_q^n$ contains an \emph{affine copy} of $B \subseteq \mathbb F_q^m$ if there is a non-degenerate affine homomorphism $B \to A$. If $\mathcal B = \{B_i\}_{i \in I}$ is a family of affine configurations $B_i \subseteq \mathbb F_q^{m_i}$, we say that $A$ is $\mathcal B$-free if $A$ contains no affine copy of any $B_i$. The largest size $\exaff(n,\mathcal B)$ of an affine $\mathcal B$-free subset of $\mathbb F_q^n$ is called the $n$-th \emph{affine extremal number} of $\mathcal B$. If $\mathcal B = \{B\}$, we write $\exaff(n,\{B\}) = \exaff(n,B)$.

For $A \subseteq \mathbb F_q^n$ and $B \subseteq \mathbb F_q^m$, define the \emph{product} of $A$ and $B$ to be the set
\[A \times B := \{(x,y) \in \mathbb F_q^{n+m} : x \in A, y \in B\}.\]
The affine and linear structures of $A \times B$ are determined by the respective affine and linear structures of $A$ and $B$.

\begin{figure}[ht]
  \centering
  \tikzstyle{every path}=[thick]
  \tikzstyle{label}=[draw=none, fill=none]
  \begin{tikzpicture}
    \foreach \a in {0,1}
        \foreach \b in {0,1}
            \foreach \c in {0,1}
                \node[circle, draw, fill=black, inner sep=0pt, minimum width=1pt](F23.\a\b\c) at (3*\a + 1.2*\b, 1.2*\c +6.2) {};
    \foreach \a in {0,1}
        \foreach \b in {0,1}
            \foreach \c in {0,1}
                \node (l.F23.\a\b\c) at (3*\a + 1.2*\b, 1.2*\c +6.2-0.4) {\a\b\c};
    \draw (-1,-1+6.2) -- (5.2,-1+6.2) -- (5.2,1.8+6.2) -- (-1,1.8+6.2) -- (-1,-1+6.2);
    \foreach \a/\b/\c in {1/0/0, 1/1/0, 0/1/0, 0/0/1}
        \node[circle, draw, fill=black, inner sep=0pt, minimum width=4pt](A.\a\b\c) at (3*\a + 1.2*\b, 1.2*\c+6.2) {};
    \node[scale=1.5] (A) at (2.1,6.2+ 0.4) {$A$};

    \foreach \d in {0,1}
        \foreach \e in {0,1}
            \node[circle, draw, fill=black, inner sep=0pt, minimum width=1pt](F22.\d\e) at (1.2*\d + 8, 1.2*\e +6.2) {};
    \foreach \d in {0,1}
        \foreach \e in {0,1}
            \node(l.F22.\d\e) at (1.2*\d + 8, 1.2*\e +6.2-0.4) {\d\e};
    \draw (-1+8,-1+6.2) -- (2.2+8,-1+6.2) -- (2.2+8,1.8+6.2) -- (-1+8,1.8+6.2) -- (-1+8,-1+6.2);
    \foreach \d/\e in {0/0, 1/0, 0/1}
        \node[circle, draw, fill=black, inner sep=0pt, minimum width=4pt](B.\d\e) at (1.2*\d + 8, 1.2*\e +6.2) {};
    \node[scale=1.5] (B) at (0.6+8,6.2+ 0.4) {$B$};

    \foreach \a in {0,1}
        \foreach \b in {0,1}
            \foreach \c in {0,1}
                \foreach \d in {0,1}
                    \foreach \e in {0,1}
                        \node[circle, draw, fill=black, inner sep=0pt, minimum width=1pt](F25.\a\b\c\d\e) at (7*\a + 3*\b + 1.2*\d,3*\c + 1.2*\e) {};
    \draw (-1,-1.2) -- (12.2,-1.2) -- (12.2,5) -- (-1,5) -- (-1,-1.2);
    \foreach \a in {0,1}
        \foreach \b in {0,1}
            \foreach \c in {0,1}
                \foreach \d in {0,1}
                    \foreach \e in {0,1}
                        \node (l.F25.\a\b\c\d\e) at (7*\a + 3*\b + 1.2*\d,3*\c + 1.2*\e - 0.4) {\a\b\c\d\e};
    \foreach \a/\b/\c in {1/0/0, 1/1/0, 0/1/0, 0/0/1}
        \foreach \d/\e in {0/0, 0/1, 1/0}
            \node[circle, draw, fill=black, inner sep=0pt, minimum width=4pt](AB.\a\b\c\d\e) at (7*\a + 3*\b + 1.2*\d,3*\c + 1.2*\e) {};
    \node[scale=1.5] (AB) at (5.6,2) {$A \times B$};
    \foreach \a/\b/\c in {1/0/0, 1/1/0, 0/1/0, 0/0/1}
        \draw (7*\a + 3*\b + 0.6,3*\c + 0.4) circle (1.5);

  \end{tikzpicture}
  \caption{Example showing the product of configurations $A \subseteq \mathbb F_2^3$ and $B \subseteq \mathbb F_2^2$.
  \label{fig:bound diagram}}
\end{figure}
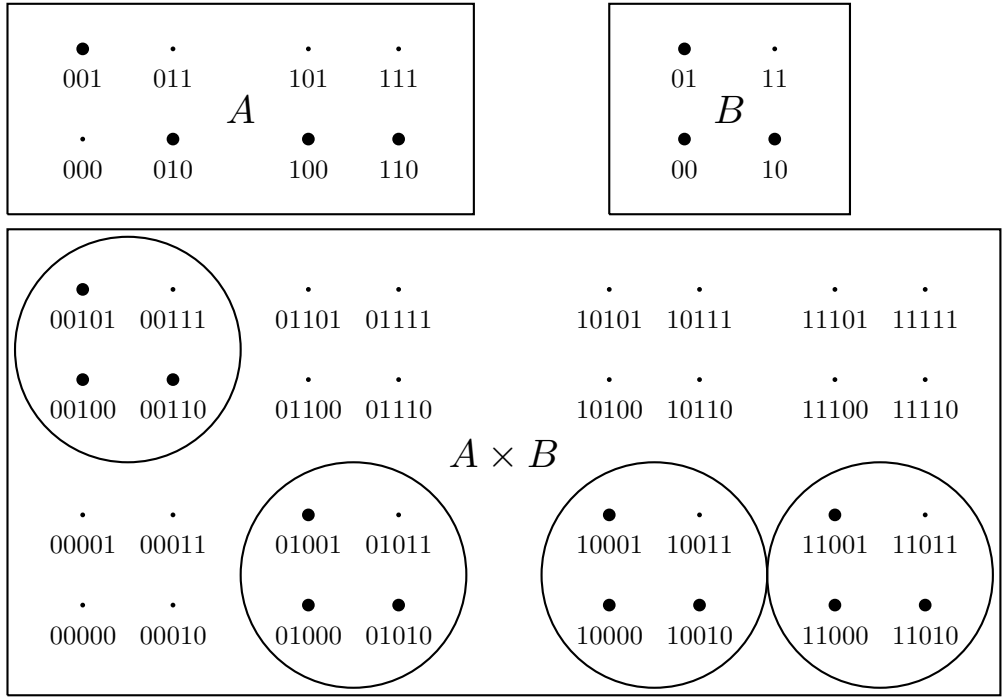  

For $A \subseteq \mathbb F_q^n$, define the \emph{direction set} of $A$ to be the set
\[A^{\to} := \{d \in \mathbb F_q^n : \text{there exists $x \in \mathbb F_q^n$ such that } x + \lambda d \in A \text{ for all } \lambda \in \mathbb F_q\}.\]
Note that for $q=2$, $A^{\to}$ is just the sumset $A + A = \{x + y : x,y \in A\}$.
Also note that the linear structure of $A^{\to}$ is entirely determined by the affine structure of $A$ since, for any $x \in \mathbb F_q^n$, the translate $A + x$ has the same direction set as $A$ does, and any linear isomorphism applied to $A$ will also preserve the linear structure of $A^\to$.
Additionally, it should be clear from the definitions that for any $A \subseteq \mathbb F_q^n, B \subseteq \mathbb F_q^m$, we have
\[(A \times B)^\to = A^\to \times B^\to.\]

We define a few more linear and affine invariants. For nonempty $A \subseteq \mathbb F_q^n$, let $\omega(A)$ denote the dimension of the largest linear subspace of $\mathbb F_q^n$ contained in $A \cup \{0\}$. Define $\omega^\to(A) := \omega(A^\to)$, which is determined by the affine structure of $A$. Let $\omegaaff(A)$ denote the dimension of the largest affine subspace of $\mathbb F_q^n$ contained in $A$. 
We also have
\[\omega(A) \leq \omegaaff(A) \leq \omega^\to(A)\]
since any linear subspace is an affine subspace, and the direction set of any affine subspace is a linear subspace of the same dimension. The following proposition shows that these invariants interact well with the product operation. Part (b) will be especially useful for our purposes.

\begin{proposition} \label{prop:product general}
Let $A \subseteq \mathbb F_q^n$ and $B \subseteq \mathbb F_q^m$ be nonempty. Then we have the following.
\begin{enumerate}[label=(\alph*)]
\item If $0 \in A$ and $0 \in B$, then $\omega(A \times B) = \omega(A) + \omega(B)$.
\item $\omega^\to(A \times B) = \omega^\to(A) + \omega^\to(B)$.
\item $\omegaaff(A \times B) = \omegaaff(A) + \omegaaff(B)$.
\item If $0 \in A$ and $0 \in B$, then $\rank(A \times B) = \rank(A) + \rank(B)$.
\item $\rankaff(A \times B) = \rankaff(A) + \rankaff(B) - 1$.
\end{enumerate}
\end{proposition}

\begin{proof}
We first prove (a). Let $k = \omega(A)$ and $\ell = \omega(B)$. Since $A$ contains a linear copy of $\mathbb F_q^k$ and $B$ contains a linear copy of $\mathbb F_q^\ell$, $A \times B$ contains a linear copy of $\mathbb F_q^k \times \mathbb F_q^\ell = \mathbb F_q^{k + \ell}$. Thus $\omega(A \times B) \geq k + \ell$.

Conversely, suppose that $W$ is a linear subspace in $A \times B$ with basis $\{(x_1,y_1), \ldots, (x_t, y_t)\}$, with $t = \omega(A \times B)$. Then $\spn\{x_1, \ldots, x_t\} \subseteq A$ and $\spn\{y_1, \ldots, y_t\}\subseteq B$, which means that $\rank\{x_1, \ldots, x_t\} \leq k$ and $\rank\{y_1, \ldots, y_t\} \leq \ell$. Now $W$ is contained in the $\leq k + \ell$ dimensional subspace of $\mathbb F_q^{n+m}$ spanned by $(x_1, 0), \ldots, (x_t, 0), (0, y_1), \ldots, (0,y_t)$, so 
\[\omega(A \times B) = \dim W \leq k + \ell.\]

Part (b) now follows. Since $0 \in A^\to$ and $0 \in B^\to$, we have by part (a) that
\begin{align*}
\omega((A \times B)^\to) &= \omega(A^\to \times B^\to) = \omega(A^\to) + \omega(B^\to).
\end{align*}

Now we prove (c). First, let $U$ be an affine space in $A$ of dimension $\omegaaff(A)$, and let $W$ be an affine space in $B$ of dimension $\omegaaff(B)$. Let $x \in U$ and $y \in W$. By translating $A$ by $-x$ and $B$ by $-y$, we may assume that $0 \in A$ and that $\omegaaff(A) = \omega(A)$, and also that $0 \in B$ and that $\omegaaff(B) = \omega(B)$. Now by part (a),
\[\omegaaff(A \times B) \geq \omega(A \times B) = \omega(A) + \omega(B) = \omegaaff(A) + \omegaaff(B).\]

Conversely, let $V$ be an affine space in $A \times B$ of dimension $\omegaaff(A \times B)$, and let $z' = (x',y') \in V$. Then $x' \in A$ and $y' \in B$, so by translating $A$ by $-x'$ and $B$ by $-y'$, we can assume that $0 \in A$, that $0 \in B$, and that $V$ is a linear space in $A \times B$. Thus by part (a),
\[\omegaaff(A \times B) = \omega(A \times B) = \omega(A) + \omega(B) \leq \omegaaff(A) + \omegaaff(B).\]

For part (d), note that, given a linear bases $\{x_1, \ldots, x_k\}$ for $A$ and a linear basis $\{y_1, \ldots, y_\ell\}$ for $B$, the set $\{(x_1, 0), \ldots, (x_k, 0), (0, y_1), \ldots, (0, y_\ell)\}$ is a linear basis for $A \times B$.

Finally, we consider part (e). By translating, we can assume $0 \in A$ and $0 \in B$, which implies $0 \in A \times B$ as well. Then $\rank(A) = \rankaff(A) - 1$, $\rank(B) = \rankaff(B) - 1$, and $\rank(A \times B) = \rankaff(A \times B) - 1$. Now by part (d), we have
\begin{align*}
\rankaff(A \times B) &= \rank(A \times B) + 1 \\
&= \rank(A) + \rank(B) + 1 \\
&= \rankaff(A) - 1 + \rankaff(B) - 1 + 1 \\
&= \rankaff(A) + \rankaff(B) - 1. \tag*{\qedhere}
\end{align*}
\end{proof}

\section{Homomorphic Supersaturation}\label{sec:homomorphic supersaturation}
We first prove a simple lemma that shows that the number of degenerate affine homomorphisms $B \to A$ is small compared to the total number of affine homomorphisms $B \to A$.

\begin{lemma}\label{lem:non-injective count}
    Let $B \subseteq \mathbb F_q^m$ and $A \subseteq \mathbb F_q^n$ be affine configurations, with $B$ nonempty. Write $r = \rankaff(B) \geq 1$, $N = q^n$, and $|A| = \alpha N$. Then the number of degenerate affine homomorphisms $B \to A$ is less than $(q \alpha N)^{r-1}$.
\end{lemma}
\begin{proof}
    If $\{x_0, \ldots, x_{r-1}\}$ is an affine basis for $B$, then an affine homomorphism $f : B \to A$ is degenerate iff $\{f(x_0), \ldots, f(x_{r-1})\} \subseteq A$ is affinely dependent. There are $(q^{r-1}-1)/(q-1) < q^{r-1}$ possible nontrivial affine relations among the $r$ elements $f(x_0), \ldots, f(x_{r-1})$, up to scaling, each of the form $\sum_{i = 0}^{r-1} \lambda_i f(x_i),$ with $\sum_{i = 0}^{r-1} \lambda_i = 0$ and some $\lambda_i \neq 0$.
    Once such a relation is established, then the entire function $f$ is determined by the values it takes on $\{x_0, \ldots, x_{r-1}\} \setminus \{x_i\}$, so there are at most $(\alpha N)^{r-1}$ such $f$'s with the given relation. Altogether, this gives the desired count.
\end{proof}

We now show that the property of $B$ being $C$-weakly Sidorenko immediately gives an upper bound on the extremal number of $B$, and that affine configurations larger than the given bound have supersaturation of affine copies of $B$. In particular, when $B$ is Sidorenko, we have the strongest possible form of supersaturation of copies of $B$ for affine configurations $A \subseteq \mathbb F_q^n$ above a certain threshold, namely the same number asymptotically as a $p$-random subset of $\mathbb F_q^n$ with $p = |A|/q^n$.  

\begin{lemma}\label{lem:Sidorenko implies supersaturation}
    Let $B \subseteq \mathbb F_q^m$ be $C$-weakly Sidorenko, with $\rankaff(B) =: r \geq 1$. Then for every $n$,
    \[\exaff(n,B) < q^{n - (n-r+1)/(C-r+1)}.\]
    Moreover, if $A \subseteq \mathbb F_q^n$ with $|A| = Dq^{(1 - 1/(C-r+1))n}$ for some $D > 0$, then $A$ contains more than
    \[\left(1-\frac{q^{r-1}}{D^{C-r+1}}\right)\frac{\alpha^C N^r}{|\autaff(B)|}\]
    subsets affinely isomorphic to $B$, where $N =q^n$ and $|A| = \alpha N$.
\end{lemma}

\begin{proof}
    We prove the supersaturation result first. If $A \subseteq \mathbb F_q^n$ has $|A| = Dq^{(1 - 1/(C-r+1))n}$, then by \Cref{lem:non-injective count}, the number of degenerate affine homomorphisms $B \to A$ is less than
    \[(q\alpha N)^{r-1} = \frac{q^{r-1}}{D^{C-r+1}}\alpha^C N^r.\]
    Since $B$ is $C$-weakly Sidorenko, we have more than
    \[\left(1 - \frac{q^{r-1}}{D^{C-r+1}}\right)\alpha^CN^r\]
    non-degenerate affine homomorphisms $B \to A$. 
    For each subset $B' \subseteq A$ which is affinely isomorphic to $B$, there are exactly $|\autaff(B)|$ non-degenerate affine homomorphisms mapping $B$ onto $B'$, so we must have more than
    \[\left(1 - \frac{q^{r-1}}{D^{C-r+1}}\right)\frac{\alpha^CN^r}{|\autaff(B)|}\]
    such subsets.
    
    In particular, if 
    \[|A| \geq q^{n - (n-r+1)/(C-r+1)},\]
    then
    \[\frac{q^{r-1}}{D^{C-r+1}} \leq 1,\]
    so $A$ must contain an affine copy of $B$, giving our desired bound on $\exaff(n,B)$.
\end{proof}

We now show that the property of being weakly Sidorenko is preserved under taking products.

\begin{proof}[Proof of \Cref{thm:Sidorenko product}]
    Let $A \subseteq \mathbb F_q^n$ with density $\alpha$, and let $N = q^n$. Fix respective affine bases $\{x_0, \ldots, x_{r_1-1}\}$ and $\{y_0, \ldots, y_{r_2-1}\}$ for $B_1$ and $B_2$. For $\mathbf u = (u_1, \ldots, u_{r_1-1}) \in (\mathbb F_q^n)^{r_1-1}$, we use $\spn_{B_1}(\mathbf u)$ to denote the set
    \[\left\{\sum_{i=1}^{r_1-1} \lambda_i u_i \in \mathbb F_q^n : x_0 + \sum_{i = 1}^{r_1-1} \lambda_i(x_i-x_0) \in B_1\right\}.\]
    We similarly define
    \[\spn_{B_2}(\mathbf v) := \left\{\sum_{i=1}^{r_2-1} \lambda_i v_i \in \mathbb F_q^n : y_0 + \sum_{i = 1}^{r_2-1} \lambda_i(y_i-y_0) \in B_2\right\}\]
    for $\mathbf v = (v_1, \ldots, v_{r_2-1}) \in (\mathbb F_q^n)^{r_2-1}$, and we further define
    \[A_{\mathbf v} := \{z \in \mathbb F_q^n : z + \spn_{B_2}(\mathbf v) \subseteq A\}.\]
    Note that
    \begin{align*}
        &\homaff(B_1, A) = \#\left\{(z, \mathbf u) \in \mathbb F_q^n \times (\mathbb F_q^n)^{r_1-1} : z + \spn_{B_1}(\mathbf u) \subseteq A\right\};\\
        &\homaff(B_2, A) = \#\left\{(z, \mathbf v) \in \mathbb F_q^n \times (\mathbb F_q^n)^{r_2-1} : z + \spn_{B_2}(\mathbf v) \subseteq A\right\}.
    \end{align*}
    We can thus express $\homaff(B_1 \times B_2, A)$ as
    \begin{align*}
        \homaff(B_1 \times B_2, A) &= \#\left\{(z, \mathbf u, \mathbf v) \in \mathbb F_q^n \times (\mathbb F_q^n)^{r_1-1} \times (\mathbb F_q^n)^{r_2-1} : z + \spn_{B_1}(\mathbf u) + \spn_{B_2}(\mathbf v) \subseteq A\right\} \\
        &= \#\left\{(z, \mathbf u, \mathbf v) \in \mathbb F_q^n \times (\mathbb F_q^n)^{r_1-1} \times (\mathbb F_q^n)^{r_2-1} : z + \spn_{B_1}(\mathbf u) \subseteq A_{\mathbf v}\right\} \\
        &= \sum_{\mathbf v \in (\mathbb F_q^n)^{r_2-1}} \#\left\{(z, \mathbf u) \in \mathbb F_q^n \times (\mathbb F_q^n)^{r_1-1} : z + \spn_{B_1}(\mathbf u) \subseteq A_{\mathbf v}\right\} \\
        &= \sum_{\mathbf v \in (\mathbb F_q^n)^{r_2-1}} \homaff(B_1, A_{\mathbf v}).
    \end{align*}
    Since $B_1$ is $C_1$-weakly Sidorenko, this is at least
    \begin{align*}
        \sum_{\mathbf v \in (\mathbb F_q^n)^{r_2-1}} \left(\frac{|A_{\mathbf v}|}{N}\right)^{C_1} N^{r_1} &= N^{r_1-C_1} \sum_{\mathbf v \in (\mathbb F_q^n)^{r_2-1}} |A_{\mathbf v}|^{C_1} \\
        &\geq N^{r_1-C_1}N^{r_2-1}\left(\frac{1}{N^{r_2-1}}\sum_{\mathbf v \in (\mathbb F_q^n)^{r_2-1}} |A_{\mathbf v}|\right)^{C_1}
    \end{align*}
    by Jensen's inequality. Since $B_2$ is $C_2$-weakly Sidorenko, and $\sum_{\mathbf v \in (\mathbb F_q^n)^{r_2-1}} |A_{\mathbf v}| = \homaff(B_2, A)$, we have
    \begin{align*}
        \homaff(B_1 \times B_2, A) &\geq N^{r_1-C_1}N^{r_2-1}\left(\frac{1}{N^{r_2-1}}\homaff(B_2, A)\right)^{C_1} \\
        &\geq N^{r_1-C_1}N^{r_2-1}\left(\alpha^{C_2}N\right)^{C_1} \\
        &= \alpha^{C_1C_2} N^{r_1+r_2-1}.
    \end{align*}
    By \Cref{prop:product general}, $\rankaff(B_1 \times B_2) = r_1 + r_2-1$, so this is our desired bound, and the proof is complete.
\end{proof}

\section{Unified Proofs of \Cref{thm:Bonin Qin} and \Cref{thm:Fox Pham}} \label{sec:recovering bq and fp}

The following is the same argument used in \cite{BQ00} and \cite{FP17}, but stated in our language in a unified and generalized way. For an affine configuration $B$ and a family $\mathcal F$ of affine configurations, we use $B \times \mathcal F$ to denote the family $\{B \times F : F \in \mathcal F\}$.

\begin{lemma}\label{lem:supersaturation implies product extremal}
    Let $B \subseteq \mathbb F_q^m$ with $r := \rankaff(B) \geq 1$, and let $\mathcal F$ be any family of affine configurations. Let $n \geq r \geq 1$, $N = q^n$, and $\exaff(n, B \times \mathcal F) = \alpha N$, and let $c(B, n,\alpha)$ denote the minimum number of non-degenerate affine homomorphisms $B \to A$ for an affine configuration $A \subseteq \mathbb F_q^n$ of density $\alpha$. Then
    \[c(B, n,\alpha) \leq q^{r-1}N^{r-1}\exaff(n-r+1, \mathcal F).\]
\end{lemma}
\begin{proof}
    Let $A \subseteq \mathbb F_q^n$ be affine $(B \times \mathcal F)$-free with density $\alpha$. Let $S$ be the set of non-degenerate affine homomorphisms $B \to A$, which has size at least $c(B, n,\alpha)$ by assumption. Fix an affine basis $\{x_0, \ldots, x_{r-1}\}$ for $B$. For each $f \in S$ and for each $1 \leq i \leq r-1$, define $u_i(f) := f(x_i)-f(x_0)$, and define
    \[\mathbf u(f) := (u_1(f), \ldots, u_{r-1}(f)).\]
    Note that the components of $\mathbf u(f)$ are linearly independent elements of $\mathbb F_q^n$ since $f$ is non-degenerate. For each ordered $(r-1)$-tuple $\mathbf u = (u_1, \ldots, u_{r-1}) \in (\mathbb F_q^n)^{r-1}$ with linearly independent components, let $S_{\mathbf u} = \{f \in S : \mathbf u(f) = \mathbf u\}$. By the Pigeonhole Principle, there exists some $\mathbf u = (u_1, \ldots, u_{r-1})$ with 
    \[|S_{\mathbf u}| \geq \frac{c(B, n,\alpha)}{N^{r-1}}.\]
    
    Now we choose a linear subspace $W_{\mathbf u} \subseteq \mathbb F_q^n$ of codimension $r-1$ with $\mathbb F_q^n = W_{\mathbf u} \oplus \spn\{u_1, \ldots, u_{r-1}\}$. We take $W_{\mathbf u}^{(1)}, \ldots, W_{\mathbf u}^{(q^{r-1})}$ to be the distinct translates of $W_{\mathbf u}$, and for each $1 \leq j \leq q^{r-1}$, we define
    \[S_{\mathbf u}^{(j)} = \{f \in S_{\mathbf u} : f(x_0) \in W_{\mathbf u}^{(j)}\}.\]
    Again, by the Pigeonhole Principle, there exists some $j$ with
    \[|S_{\mathbf u}^{(j)}| \geq \frac{c(B, n,\alpha)}{q^{r-1}N^{r-1}}.\]
    
    We now define 
    \[A_{\mathbf u}^{(j)} = \left\{f(x_0) : f \in S_{\mathbf u}^{(j)}\right\} \subseteq W_{\mathbf u}^{(j)}.\]
    Note that the map $S_{\mathbf u}^{(j)} \to A_{\mathbf u}^{(j)}$ given by $f \mapsto f(x_0)$ is a bijection with inverse
    \[y \mapsto \left[f(x_i) = \left\{\begin{array}{l l}
    y & \text{if $i=0$} \\
    y + u_i & \text{otherwise}
    \end{array}\right.\right].\]
    In particular,
    \[|A_{\mathbf u}^{(j)}| \geq \frac{c(B, n,\alpha)}{q^{r-1}N^{r-1}}.\]
    On the other hand, if we have an affine copy $F'$ of some member $F \in \mathcal F$ in $A_{\mathbf u}^{(j)}$, then 
    \[G := \left\{f(x) : x \in B, f \in S_{\mathbf u}^{(j)}, f(x_0) \in F' \right\}\]
    is an affine copy of $B \times F \in B \times \mathcal F$ in $A$, contrary to assumption. Thus $A_{\mathbf u}^{(j)} \subseteq W_{\mathbf u}^{(j)}$ is $\mathcal F$-free, so we have
    \[\frac{c(B, n,\alpha)}{q^{r-1}N^{r-1}} \leq |A_{\mathbf u}^{(j)}| \leq \exaff(n-r+1, \mathcal F). \qedhere\]
\end{proof}

We now apply \Cref{lem:supersaturation implies product extremal} iteratively to recover the results of Bonin-Qin and Fox-Pham.
\begin{proof}[Proof of \Cref{thm:Bonin Qin} and \Cref{thm:Fox Pham}]
For $q \in \{2,3\}$, let  $\sigma_q$ be as in \Cref{lem:affine lines Sidorenko}. We will prove by induction on $t$ that for all $n\geq t$,
\begin{equation}\label{eq:explicit extremal bound}
    \exaff(n,\mathbb F_q^t) < q^{n-n/((\sigma_q-1)\sigma_q^{t-1})+2}.
\end{equation}
For $t=1$, let $N = q^n$, and let $A \subseteq \mathbb F_q^n$ be $\mathbb F_q^1$-free of size $\alpha N = \exaff(n,\mathbb F_q^1)$. By \Cref{lem:affine lines Sidorenko}, there are at least $\alpha^{\sigma_q}N^2$ affine homomorphisms $\mathbb F_q^1 \to A$, all of which are degenerate since $A$ is $\mathbb F_q^1$-free. But the degenerate affine homomorphisms $\mathbb F_q^1 \to A$ are precisely the constant maps, so we have $\alpha^{\sigma_q}N^2 \leq \alpha N,$ and hence
\[\alpha N \leq N^{1-1/(\sigma_q-1)}.\]

Now assume $t \geq 2$. Let $N = q^n$, let $\exaff(n,\mathbb F_q^t) = \alpha N$, and let $A \subseteq \mathbb F_q^t$ be an affine configuration of density $\alpha$. Again, by \Cref{lem:affine lines Sidorenko}, there are at least $\alpha^{\sigma_q}N^2 - \alpha N$ non-degenerate affine homomorphisms $\mathbb F_q^1 \to A$. Therefore, by \Cref{lem:supersaturation implies product extremal}, with $B = \mathbb F_q^1$, $\mathcal F = \{\mathbb F_q^{t-1}\}$, and $c(B, n, \alpha) \geq \alpha^{\sigma_q}N^2 - \alpha N$, we have
\[\alpha^{\sigma_q}N^2 - \alpha N \leq qN\exaff(n-1,\mathbb F_q^{t-1}).\]
By the inductive hypothesis, this gives
\[\alpha^{\sigma_q}N^2 - \alpha N < qNq^{n-1 - (n-1)/((\sigma_q-1)\sigma_q^{t-2})+2} = q^{2+1/((\sigma_q-1)\sigma_q^{t-2})}N^{2 - 1/((\sigma_q-1)\sigma_q^{t-2})}.\]
We can assume that $\alpha^{\sigma_q}N^2 \geq 2\alpha N$, as otherwise the claim holds already. Now we have
\[\frac 12 \alpha^{\sigma_q} N^2 < q^{2+1/((\sigma_q-1)\sigma_q^{t-2})}N^{2 - 1/((\sigma_q-1)\sigma_q^{t-2})},\]
and hence
\[\alpha N < \left(2 q^{2+1/((\sigma_q-1)\sigma_q^{t-2})}\right)^{1/\sigma_q}N^{1- 1/((\sigma_q-1)\sigma_q^{t-1})} \leq q^2 N^{1-1/((\sigma_q-1)\sigma_q^{t-1})}. \qedhere\]
\end{proof}

We observe that our supersaturation results from \Cref{sec:homomorphic supersaturation} give an alternative proof of Theorems \ref{thm:Bonin Qin} and \ref{thm:Fox Pham}. This method additionally establishes a strong supersaturation result for affine subspaces of $\mathbb F_q$ for $q \in \{2,3\}$.

\begin{theorem}\label{thm:supersaturation of affine spaces}
    For $q \in \{2,3\}$, let $\sigma_q$ be as in the statement of \Cref{lem:affine lines Sidorenko}, and let $t \geq 0$. Then for any $n$,
    \[\exaff(n, \mathbb F_q^t) < q^{n-(n-t)/(\sigma_q^t-t)}.\]
    Moreover, if $A \subseteq \mathbb F_q^n$ with $|A| = Dq^{(1 - 1/(\sigma_q^t-t))n}$ for some $D > 0$, then $A$ contains more than
    \[\left(1-\frac{q^t}{D^{\sigma_q^t-t}}\right)\frac{\alpha^{\sigma_q^t}N^{t+1}}{|\autaff(\mathbb F_q^t)|}\]
    affine $t$-spaces, where $N = q^n$ and $|A| = \alpha N$.
\end{theorem}
\begin{proof}
    The claim is simply a special case of \Cref{lem:Sidorenko implies supersaturation}. By \Cref{lem:affine lines Sidorenko}, $\mathbb F_q^1$ is $\sigma_q$-weakly Sidorenko, so by \Cref{thm:Sidorenko product}, $\mathbb F_q^t$ is $\sigma_q^t$-weakly Sidorenko, with $\rankaff(\mathbb F_q^t) = t+1$. 
\end{proof}

We take a moment to compare our supersaturation results to a supersaturation result of Gijswijt (\cite{G21} Proposition 22). His result applies to any affine configuration $B \subseteq \mathbb F_q^m$ whose $n$-th affine extremal number is bounded above by $(q^{1-\delta})^n$ for some constant $\delta > 0$. He proves that affine configurations in $\mathbb F_q^n$ with density $\alpha \gg q^{-\delta n}$ have $\Omega(\alpha^{(r-1+2\delta)/\delta}N^r)$ affine copies of $B$, where $r = \rankaff(B)$ and $N = q^n$. Our result \Cref{lem:Sidorenko implies supersaturation} improves this count to $\Omega(\alpha^{(\delta(r-1) + 1)/\delta} N^r)$ affine copies of $B$ when $B$ is $C$-weakly Sidorenko and we take $\delta = 1/(C-r+1)$. In particular, for $q \in \{2,3\}$, Gijswijt's result guarantees only $\alpha^{(1 + o(1))t\sigma_q^t}N^{t+1}$ affine $t$-spaces when $|A|$ is above the Bonin-Qin threshold (for $q=2$) or Fox-Pham threshold (for $q=3$). Now \Cref{thm:supersaturation of affine spaces} improves this to $\Omega(\alpha^{\sigma_q^t}N^{t+1})$ affine $t$-spaces, which is tight for $q = 2$ by considering a random affine configuration of density $\alpha$.

\section{Proof of \Cref{thm:diagonal Fq k colors}}
We first prove a general upper bound for the two-color Ramsey number $R_q(s,t)$ for $q \in \{2,3\}$, from which \Cref{thm:diagonal Fq k colors} is easily derived. The proof uses nothing more than \Cref{thm:Bose Burton} and our explicit forms of \Cref{thm:Bonin Qin} and \Cref{thm:Fox Pham} for bounds on $\exaff(n,\mathbb F_q^t)$.
\begin{theorem}\label{thm:diagonal Fq two colors}
    For $q \in \{2,3\}$, let $\sigma_q$ be as in \Cref{lem:affine lines Sidorenko}. For any $t \geq s \geq 2$, $R_q(s,t)$ is at most a tower of height $s$ of the form
    \[R_q(s,t) \leq  \sigma_q^{\sigma_q^{\iddots^{\sigma_q^{2t}}}}.\]
\end{theorem}
\begin{proof}
    We induct on $s$ for fixed $t$. Clearly, we have $R_q(1,t) = t$. Now for $s \geq 2$, let $r = R_q(s-1,t)$, and let $n = t\sigma_q^r$. Suppose we have a projectively determined red-blue coloring of $\mathbb F_q^n \setminus \{0\}$ with red set $R$ and blue set $B$ satisfying $\omega(R) < s$ and $\omega(B) < t$. By \Cref{thm:Bose Burton}, the fact that $\omega(B) < t$ implies that $|R| \geq q^{n-t+1} - 1$, with equality if and only if $R \cup \{0\}$ is a linear $(n-t+1)$-space. But $\omega(R) < s \leq n-t+1$, so we must have $|R| \geq q^{n-t+1}$. Also, by \Cref{thm:supersaturation of affine spaces},
    \[\exaff(n,\mathbb F_q^r) \leq q^{n-(n-r)/(\sigma_q^r-r)} \leq q^{n-n/\sigma_q^r} = q^{n-t} < |R|,\]
    so $R$ contains an affine $r$-space $A$.  Note that $0 \notin A$ since $0 \notin R$. Let $W$ be the translate of $A$ containing $0$, which is a linear $r$-space. Then by our choice of $r$ and because we've assumed $\omega(B) < t$, there exists a linear $(s-1)$-space $U' \subseteq W$ with $U' \setminus \{0\}$ entirely red. Now because the coloring is projectively determined, for any $u \in A$ and $\lambda \in \mathbb F_q \setminus \{0\}$, the set $W + \lambda u = \lambda A$ is entirely red. But then $U := \spn\{U',u\}$ is a linear $s$-space contained in $U' \cup \bigcup_{\lambda \in \mathbb F_q \setminus \{0\}} \lambda A$, so $U \setminus \{0\}$ is entirely red, a contradiction. Thus
    \[R_q(s,t) \leq n = t\sigma_q^r.\]

    By induction, $R_q(s,t)$ is at most a tower of height $s$ of the form
    \[R_q(s,t) \leq t\sigma_q^{t\sigma_q^{\iddots^{t\sigma_q^t}}}.\]
    
    To obtain the friendlier-looking bound stated in the theorem, it suffices to show that $\log_{\sigma_q}^{(s-1)}(R_q(s,t)) \leq 2t$, where $\log_b^{(k)}(x)$ denotes the $k$-th iterated logarithm (to base $b$) of $x$, defined by
    \[\log_b^{(k)}(x) := \left\{\begin{array}{l l}
    x & \text{if } k = 0, \\
    \log_b\left(\log_b^{(k-1)}(x)\right) & \text{if } k \geq 1 \text{ and } \log_b^{(k-1)}(x) > 0, \\
    -\infty & \text{if } k \geq 1 \text{ and } \log_b^{(k-1)}(x) \leq 0.
    \end{array}\right.\]
    First note that
    \[\log_{\sigma_q}(R_q(s,t)) \leq \log_{\sigma_q}\left(t\sigma_q^{t\sigma_q^{\iddots^{t\sigma_q^t}}}\right) = \log_{\sigma_q} t + t\sigma_q^{t\sigma_q^{\iddots^{t\sigma_q^t}}} \leq 2t\sigma_q^{t\sigma_q^{\iddots^{t\sigma_q^t}}},\]
    where the height of the tower on the right is $s-1$. Now applying the logarithm again gives
    \[\log_{\sigma_q}^{(2)}\left(R_q(s,t)\right) \leq \log_{\sigma_q}(2t) + t\sigma_q^{t\sigma_q^{\iddots^{t\sigma_q^t}}} \leq 2t\sigma_q^{t\sigma_q^{\iddots^{t\sigma_q^t}}},\]
    where the height is now $s-2$. Continuing in this fashion, we obtain $\log_{\sigma_q}^{(s-2)}(R_q(s,t)) \leq 2t\sigma_q^t$, and thus 
    \[\log_{\sigma_q}^{(s-1)}(R_q(s,t)) \leq \log_{\sigma_q}(2t) + t \leq 2t\]
    since $t \geq 2$.
\end{proof}

\begin{proof}[Proof of \Cref{thm:diagonal Fq k colors}]
    For $q \in \{2,3\}$, let $\sigma_q$ be as in \Cref{lem:affine lines Sidorenko}. We induct on $k$. The base case $k=2$ is given by \Cref{thm:diagonal Fq two colors}.
    
    For $k \geq 3$, we use the simple recurrence
    \[R_q(t_1, \ldots, t_k) \leq R_q(t_1,\ldots, t_{k-2}, R_q(t_{k-1}, t_k)).\]
    Indeed, consider a partition $\mathbb F_q^n \setminus \{0\} = B_1 \cup \cdots \cup B_k$ with $n = R_q(t_1, \ldots, t_{k-2}, R_q(t_{k-1}, t_k))$, where each set $B_i$ is projectively determined. If $\omega(B_i) < t_i$ for all $i \leq k-2$, then we must have $\omega(B_{k-1} \cup B_k) \geq R_q(t_{k-1}, t_k)$, and so $\omega(B_i) \geq t_i$ for some $i \geq k-1$. 
    
    With this observation, we obtain by the inductive hypothesis that
    \[R_q(t_1, \ldots, t_k) \leq R_q(t_1,\ldots, t_{k-2}, R_q(t_{k-1}, t_k)) \leq \sigma_q^{\sigma_q^{\iddots^{\sigma_q^{3R_q(t_{k-1}, t_k)}}}},\]
    where the height of the tower is $\sum_{i = 1}^{k-2}(t_i-1) + 1$.
    Now by \Cref{thm:diagonal Fq two colors}, 
    \[\log_{\sigma_q}^{(t_{k-1}-1)}(R_q(t_{k-1}, t_k)) \leq 2t_k,\]
    which implies
    \[\log_{\sigma_q}^{(t_{k-1}-1)}(3R_q(t_{k-1}, t_k)) \leq 2t_k + \log_{\sigma_q} 3 \leq 3t_k\]
    since $t_k \geq 2$. This completes the inductive step.
\end{proof}

\section{A Reformulation of $R_q(2,t)$}

We now reformulate the off-diagonal Ramsey problem as an affine extremal problem. We look at the $\mathbb F_2$ case first for the sake of exposition. Consider the \emph{sumset} of $A \subseteq \mathbb F_2^n$, defined as 
\[A + A := \{x + y : x,y \in A\},\]
and let $m_2(t)$ be the minimum $n$ such that every set $A \subseteq \mathbb F_2^n$ of size at least $2^{n-t+1}$ satisfies $\omega(A+A) \geq t$; that is, $A+A$ contains a linear $t$-space. Nelson and Nomoto \cite{NN21} observed that $m_2(t)$ is an upper bound for $R_2(2,t)$ for all $t \geq 2$ (see \Cref{lem:extremal implies Ramsey} for the argument). One way to bound $m_2(t)$ from above is via the following theorem of Sanders \cite{S10}.

\begin{theorem}[Sanders]\label{thm:Sanders}
    Let $A$ be a subset of $\mathbb F_2^n$ of density $\alpha < 1/2$. Then
    \[\omega(A+A) \geq n - \left\lceil n/\log_2 \frac{2-2\alpha}{1-2\alpha} \right\rceil.\]
\end{theorem}

Taking $\alpha = 2^{1-t}$ and $n = (t+1)2^t$, and noting that $n - \left\lceil n/\log_2 \frac{2-2\alpha}{1-2\alpha} \right\rceil \geq \alpha n/2 - 1 = t$ for this choice of parameters, \Cref{thm:Sanders} gives $m_2(t) \leq n$. This is how \Cref{thm:Nelson Nomoto} is proved in \cite{NN21}.

Alternatively, we can take an affine extremal approach to bound $m_2(t)$, based on the simple observation that $\omega(A+A) \geq t$ if and only if $A$ contains an affine copy $B'$ of some affine configuration $B$ with $\omega(B+B) \geq t$. Indeed, as noted in \Cref{sec:preliminaries}, $\omega(B+B) = \omega^\to(B)$ is entirely determined by the affine structure of $B$, so $\omega(B'+B') = \omega(B+B)$. Therefore, if we define
\[\mathcal B_2^t := \{B \subseteq \mathbb F_2^m : m \geq 1,\, \omega(B+B) \geq t\},\]
then we have the alternative description of $m_2(t)$ as the minimum $n$ such that $\exaff(n,\mathcal B_2^t) < 2^{n-t+1}$. We see that this is finite by \Cref{thm:Furstenberg Katznelson}, and in fact, \Cref{thm:Bonin Qin} immediately implies an improvement of \Cref{thm:Nelson Nomoto} by a constant factor. Note that any set $A$ which \textit{properly} contains an affine $(t-1)$-space has $\omega(A+A) \geq t$, so using the explicit bound in (\ref{eq:explicit extremal bound}), we have
\[\exaff(n, \mathcal B_2^t) \leq \exaff(n, \mathbb F_2^{t-1}) < 2^{(1-2^{2-t})n + 2}\]
for $n \geq t$.
In particular, if $n = (t+1)2^{t-2}$, then $\exaff(n, \mathcal B_2^t) < 2^{n-t+1}$, so
\[R_2(2,t) \leq m_2(t) \leq (t+1)2^{t-2}.\]
We obtain further improvements on $R_2(2,t)$ by finding better upper bounds for $\exaff(n, \mathcal B_2^t)$.

More generally, we define for an arbitrary finite field $\mathbb F_q$
\[\mathcal B_q^t := \{B \subseteq \mathbb F_q^m : m \geq 1, \omega^\to(B) \geq t\},\]
and we define $m_q(t)$ to be the minimum $n$ such that $\exaff(n, \mathcal B_q^t) < q^{n-t+1}$. Equivalently, $m_q(t)$ is the minimum $n$ such that $\omega^\to(A) \geq t$ for every $A \subseteq \mathbb F_q^n$ of size at least $q^{n-t+1}$. We have the following.

\begin{lemma}\label{lem:extremal implies Ramsey}
    Let $\mathbb F_q$ be any finite field. Then $R_q(2,t) \leq m_q(t)$ for all $t \geq 2$.
\end{lemma}
\begin{proof}
    Let $n = m_q(t)$. First, we show that $n \geq t+1$. Let $H$ be a linear hyperplane in $\mathbb F_q^n$, which satisfies $\omega^\to(H) = n-1$. Since $|H| = q^{n-1} \geq q^{n-t+1}$, we have $\omega^\to(H) \geq t$ by our choice of $n$.
    
    Now suppose we have a projectively determined red-blue coloring of $\mathbb F_q^n \setminus \{0\}$ with red set $R$ and blue set $B$ satisfying $\omega(R) < 2$ and $\omega(B) < t$. Since $\omega(B) < t$, we have by \Cref{thm:Bose Burton} that $|R| \geq q^{n-t+1}-1$, with equality iff $R \cup \{0\}$ is a linear $(n-t+1)$-space. But $n-t+1 \geq 2 > \omega(R)$, so we can't have equality, and hence $|R| \geq q^{n-t+1}$. By our choice of $n$, $\omega^\to(R) \geq t > \omega(B)$, so there exists some nonzero $d \in R^\to \setminus B$. That is, $d \in R^\to \cap R$. Let $a \in R$ be such that $a + \lambda d \in R$ for every $\lambda \in \mathbb F_q$. Note that $a$ and $d$ are linearly independent since $0 \notin R$. Therefore, since $R$ is projectively determined, $\spn\{a,d\}$ is a linear $2$-space contained in $R \cup \{0\}$, contradicting that $\omega(R) < 2$.
\end{proof}

We can now use the machinery from \Cref{sec:homomorphic supersaturation} to prove \Cref{thm:off-diagonal Fq}.
\begin{proof}[Proof of \Cref{thm:off-diagonal Fq}]
    Let $t \geq 1$, and let $k = \lceil t/4 \rceil$. Consider the affine configuration $C_6 \subseteq \mathbb F_2^4$, defined prior to \Cref{lem:affine lines Sidorenko}. Suppose we have two pairs $\{x, y\}, \{x', y'\} \in \binom{C_6}{2}$ with $x + y = x' + y'$. Since every $4$ distinct elements of $C_6$ are affinely independent, we must have that $x,y,x',y'$ are not distinct. Then $x=x'$ without loss of generality, which implies $y=y'$ as well. Thus we have $\binom 62 = 15$ distinct nonzero sums $x + y \in \mathbb F_2^4$ for $x,y \in C_6$ with $x \neq y$, which means that $C_6 + C_6 = \mathbb F_2^4$, and we have $\omega^\to(C_6) = 4$. By \Cref{prop:product general}, $\omega^\to(C_6^k) = 4k \geq t$ and $\rankaff(C_6^k) = 4k+1$. Recall that $\mathcal B_2^t = \{B : \omega^{\to}(B) \geq t\}$, so $C_6^k \in \mathcal B_2^t$. Furthermore, by \Cref{lem:affine lines Sidorenko} and \Cref{thm:Sidorenko product}, $C_6^k$ is Sidorenko, so by \Cref{lem:Sidorenko implies supersaturation},
    \[\exaff(n,\mathcal B_2^t) \leq \exaff(n,C_6^k) < 2^{n-(n-4k)/(6^k-4k)} = 2^{n-t+1}\]
    for $n = (t-1)(6^k-4k)+4k$. Thus
    \[R_2(2,t) \leq m_2(t) \leq (t-1)(6^k-4k)+4k = O\big(t 6^{t/4}\big)\]
    by \Cref{lem:extremal implies Ramsey}.

    Similarly, since $\omega^\to(\mathbb F_3^t) = t$, we have by \Cref{thm:supersaturation of affine spaces} that
    \[\exaff(n, \mathcal B_3^t) \leq \exaff(n,\mathbb F_3^t) < 2^{n-(n-t)/(C_0^t-t)} = 2^{n-t+1}\]
    for $n = (t-1)(C_0^t-t)+t$, with $C_0 \approx 13.901$ as in \Cref{thm:Fox Pham}. Thus
    \[R_3(2,t) \leq m_3(t) \leq (t-1)(C_0^t-t)+t = O\big(tC_0^t\big)\]
    by \Cref{lem:extremal implies Ramsey}.
\end{proof}

We remark that the leading constants in our bounds for \Cref{thm:off-diagonal Fq} are not optimized. Bounding the extremal numbers of $\mathcal B_2^t$ and $\mathcal B_3^t$ via iterative application of \Cref{lem:supersaturation implies product extremal} gives the best results, but the computations are slightly more cumbersome.

We state a generalization of this argument, which can be used to further improve our off-diagonal Ramsey bounds by establishing homomorphic supersaturation of affine configurations. Unfortunately, this technique by itself can never give a subexponential bound for $R_q(2,t)$. Indeed, for any affine configuration $B \subseteq \mathbb F_q^m$ with $\omega^\to(B) \geq 1$, the map $f : B^2 \to \mathbb F_q^m$ given by $f(x,y) = y-x$ has $B^\to$ in its image so $|B^\to| \leq |B|^2$. Also, for any linear configuration $A$, we have $|A| \geq q^{\omega(A)}$, and hence 
\[|B|^{1/\omega^\to(B)} \geq |B^\to|^{1/(2\omega^\to(B))} \geq q^{1/2}.\]
Thus the best possible upper bound that can come directly from \Cref{thm:Sidorenko implies Ramsey} is $\Omega\big(tq^{t/2}\big)$.

\begin{theorem}\label{thm:Sidorenko implies Ramsey}
    Suppose that $B \subseteq \mathbb F_q^m$ is $C$-weakly Sidorenko, and let $p = \omega^\to(B) \geq 1$. Then as $t \to \infty$,
    \[R_q(2,t) = O\left(tC^{t/p}\right).\]
\end{theorem}
\begin{proof}
    Let $n \geq t \geq 2$, and let $k = \lceil t/p \rceil$ and $r = \rankaff(B)$. By \Cref{thm:Sidorenko product} and \Cref{prop:product general}, $B^k$ is $C^k$-weakly Sidorenko with $\omega^\to(B^k) = kp \geq t$ and $\rankaff(B^k) = (r-1)k+1$. Therefore, by \Cref{lem:Sidorenko implies supersaturation}, 
    \[\exaff(n, \mathcal B_q^{t}) \leq \exaff(n,B^k) < q^{n-(n-(r-1)k)/(C^k-(r-1)k)}\]
    If we take
    \[n = (t-1)(C^k-(r-1)k)+(r-1)k = O(tC^{t/p}),\]
    then $\exaff(n,\mathcal B_q^t) < q^{n-t+1}.$ By \Cref{lem:extremal implies Ramsey}, $R_q(2,t) \leq m_q(t) \leq n$.
\end{proof}

\section{Concluding Remarks}
We believe $m_q(t) = \min\{n : \exaff(n, \mathcal B_q^t) < q^{n-t+1}\}$ to be polynomial in $t$ for all $q$, which would imply that $R_q(2,t)$ is also polynomial by \Cref{lem:extremal implies Ramsey}. For $q=2$, this was asked by Peter Nelson~\cite{PN2022} in the second Barbados graph theory workshop 2022 (Problem 17), and this remains open. For $q \neq 2,3$, it is unknown whether $m_q(t)$ is even bounded by an exponential function. Such a bound would follow immediately from exponential improvements on the affine extremal number of $\mathbb F_q^1$ by \Cref{lem:supersaturation implies product extremal}, combined with the aforementioned supersaturation result of Gijswijt (\cite{G21}, Proposition 22).
In particular, if it is true that $\exaff(n,\mathbb F_q^1) \leq (q^{1-\delta})^n$ for some $\delta > 0$, then we immediately obtain
\[R_q(2,t) \leq m_q(t) = O\left(t(2 + 1/\delta)^t\right).\]

It is also worth mentioning the natural relationship of affine extremal numbers to \emph{affine Ramsey numbers}. We use $\Raff{q}(t_1, \ldots, t_k)$ to denote the minimum $n$ such that for every $k$-coloring $f : \mathbb F_q^n \to [k]$ of the points of $\mathbb F_q^n$, there exist $i \in [k]$ and an affine subspace $U \subseteq \mathbb F_q^n$ of dimension $t_i$, such that $U$ is monochromatic in color $i$. If $t_1 = \cdots = t_k = t$, we write $\Raff{q}(t_1, \ldots, t_k) = \Raff{q}(t;k)$. Such Ramsey numbers clearly exist by \Cref{thm:Furstenberg Katznelson} since the majority color class, say color $i$, has size at least $q^n/k$, which is greater than $\exaff(n,\mathbb F_q^{t_i})$ for large $n$. In fact, any general upper bound for $\exaff(n, \mathbb F_q^t)$ immediately implies upper bounds for affine Ramsey numbers. For $q \in \{2,3\}$, Theorems \ref{thm:Bonin Qin} and \ref{thm:Fox Pham} give
\[\Raff{q}(t;k) \leq (\log_2 k)\sigma_q^t \qquad \text{for all $k \geq 2$, $t \geq 1$;}\]
\[\Raff{q}(s,t) \leq (\log_q \sigma_q)(\sigma_q-1)\sigma_q^{s-1} t \qquad \text{for all $s$ fixed, $t$ large,}\]
where $\sigma_q$ is as in \Cref{lem:affine lines Sidorenko}. Upper bounds on Hales-Jewett numbers (see \cite{S88}, for example) also imply upper bounds on affine Ramsey numbers for general $q$, though these are of a much larger order of growth. For lower bounds, straightforward applications of the Lov\'asz Local Lemma give the following:
\[\Raff{q}(t;k) \geq (\log_q k)\frac{q^t}{t} \qquad \text{for all $k$ fixed, $t$ large};\]
\[\Raff{q}(s,t) \geq \left(\frac{q^s - 1}{s} - o(1)\right) t \qquad \text{for all $s$ fixed, 
 as $t \to \infty$.} \]
 It would be interesting to see new methods develop for obtaining upper bounds on affine Ramsey numbers.

\section{Acknowledgements}
We would like to thank Peter Nelson for posing the problem of improving upper bounds on $R_2(2,t)$ in the Barbados graph theory workshop 2022 which motivated most of our work here. The authors would like to also thank Tom Sanders for stimulating discussions on the topic.
\bibliography{literature}

\begin{bibdiv}
\begin{biblist}

\bib{PN2022}{misc}{
       title={{Open problems for the second 2022 Barbados workshop}},
  note={\url{https://web.math.princeton.edu/~tunghn/2022openproblems.pdf}},
}

\bib{A16}{book}{
      author={Alon, Noga},
      author={Spencer, Joel~H},
       title={The probabilistic method},
   publisher={John Wiley \& Sons},
        date={2016},
}

\bib{A22a}{article}{
      author={Altman, Daniel},
       title={{Local aspects of the Sidorenko property for linear equations}},
        date={2022},
     journal={arXiv:2210.17493},
}

\bib{A22b}{article}{
      author={Altman, Daniel},
       title={{On a question of Alon}},
        date={2022},
     journal={arXiv:2210.13515},
}

\bib{BQ00}{article}{
      author={Bonin, Joseph~E.},
      author={Qin, Hongxun},
       title={Size functions of subgeometry-closed classes of representable
  combinatorial geometries},
        date={2000},
     journal={Discrete Mathematics},
      volume={224},
      number={1-3},
       pages={37\ndash 60},
}

\bib{BB66}{article}{
      author={Bose, Raj~Chandra},
      author={Burton, RC},
       title={{A characterization of flat spaces in a finite geometry and the
  uniqueness of the Hamming and the MacDonald codes}},
        date={1966},
     journal={Journal of Combinatorial Theory},
      volume={1},
      number={1},
       pages={96\ndash 104},
}

\bib{CCS07}{article}{
      author={Cameron, Peter~J.},
      author={Cilleruelo, Javier},
      author={Serra, Oriol},
       title={On monochromatic solutions of equations in groups},
        date={2007},
     journal={Revista Matem{\'a}tica Iberoamericana},
      volume={23},
      number={1},
       pages={385\ndash 395},
}

\bib{CLP17}{article}{
      author={Croot, Ernie},
      author={Lev, Vsevolod~F.},
      author={Pach, P{\'e}ter~P{\'a}l},
       title={Progression-free sets in are exponentially small},
        date={2017},
     journal={Annals of Mathematics},
       pages={331\ndash 337},
}

\bib{EG17}{article}{
      author={Ellenberg, Jordan~S.},
      author={Gijswijt, Dion},
       title={On large subsets of with no three-term arithmetic progression},
        date={2017},
     journal={Annals of Mathematics},
       pages={339\ndash 343},
}

\bib{FL17}{inproceedings}{
      author={Fox, Jacob},
      author={Lov{\'a}sz, L{\'a}szl{\'o}Mikl{\'o}s},
       title={{A tight bound for Green's arithmetic triangle removal lemma in
  vector spaces}},
organization={SIAM},
        date={2017},
   booktitle={{Proceedings of the Twenty-Eighth Annual ACM-SIAM Symposium on
  Discrete Algorithms}},
       pages={1612\ndash 1617},
}

\bib{FP17}{article}{
      author={Fox, Jacob},
      author={Pham, Huy~Tuan},
       title={{Popular progression differences in vector spaces II}},
        date={2019},
     journal={Discrete Analysis},
}

\bib{FPZ21}{article}{
      author={Fox, Jacob},
      author={Pham, Huy~Tuan},
      author={Zhao, Yufei},
       title={Common and {Sidorenko} linear equations},
        date={2021},
     journal={The Quarterly Journal of Mathematics},
      volume={72},
      number={4},
       pages={1223\ndash 1234},
}

\bib{FK85}{article}{
      author={Furstenberg, Hillel},
      author={Katznelson, Yitzhak},
       title={{An ergodic Szemer{\'e}di theorem for IP-systems and
  combinatorial theory}},
        date={1985},
     journal={Journal d’Analyse Math{\'e}matique},
      volume={45},
      number={1},
       pages={117\ndash 168},
}

\bib{FK91}{article}{
      author={Furstenberg, Hillel},
      author={Katznelson, Yitzhak},
       title={{A density version of the Hales-Jewett theorem}},
        date={1991},
     journal={Journal d’Analyse Math{\'e}matique},
      volume={57},
      number={1},
       pages={64\ndash 119},
}

\bib{GN15}{article}{
      author={Geelen, Jim},
      author={Nelson, Peter},
       title={{An analogue of the Erd{\H{o}}s-Stone theorem for finite
  geometries}},
        date={2015},
     journal={Combinatorica},
      volume={35},
       pages={209\ndash 214},
}

\bib{G21}{article}{
      author={Gijswijt, Dion},
       title={Excluding affine configurations over a finite field},
        date={2021},
     journal={arXiv:2112.12620},
}

\bib{GLR72}{article}{
      author={Graham, R.~L.},
      author={Leeb, K.},
      author={Rothschild, B.~L.},
       title={Ramsey's theorem for a class of categories},
        date={1972},
        ISSN={0001-8708},
     journal={Advances in Math.},
      volume={8},
       pages={417\ndash 433},
         url={https://doi.org/10.1016/0001-8708(72)90005-9},
      review={\MR{306010}},
}

\bib{G05}{article}{
      author={Green, Ben},
       title={{A Szemer{\'e}di-type regularity lemma in abelian groups, with
  applications}},
        date={2005},
     journal={Geometric \& Functional Analysis GAFA},
      volume={15},
      number={2},
       pages={340\ndash 376},
}

\bib{HJ63}{article}{
      author={Hales, A.~W.},
      author={Jewett, R.~I.},
       title={Regularity and positional games},
        date={1963},
        ISSN={0002-9947},
     journal={Trans. Amer. Math. Soc.},
      volume={106},
       pages={222\ndash 229},
         url={https://doi.org/10.2307/1993764},
      review={\MR{143712}},
}

\bib{KLM22}{article}{
      author={Kam\v{c}ev, Nina},
      author={Liebenau, Anita},
      author={Morrison, Natasha},
       title={On uncommon systems of equations},
        date={2022},
     journal={arXiv:2106.08986},
}

\bib{KLM23}{article}{
      author={Kam\v{c}ev, Nina},
      author={Liebenau, Anita},
      author={Morrison, Natasha},
       title={{Towards a characterization of Sidorenko systems}},
        date={2023},
     journal={The Quarterly Journal of Mathematics},
       pages={haad013},
}

\bib{KLP22}{inproceedings}{
      author={Kr\'al, Daniel},
      author={Lamaison, Ander},
      author={Pach, P{\'e}ter~P{\'a}l},
       title={Common systems of two equations over the binary field},
organization={Editorial de la Universidad de Cantabria},
        date={2022},
   booktitle={Discrete mathematics days 2022},
       pages={169\ndash 173},
}

\bib{NN21}{article}{
      author={Nelson, Peter},
      author={Nomoto, Kazuhiro},
       title={The structure of claw-free binary matroids},
        date={2021},
     journal={Journal of Combinatorial Theory, Series B},
      volume={150},
       pages={76\ndash 118},
}

\bib{polymath12}{article}{
      author={Polymath, D. H.~J.},
       title={{A new proof of the density Hales-Jewett theorem}},
        date={2012},
     journal={Annals of Mathematics},
       pages={1283\ndash 1327},
}

\bib{RTST06}{article}{
      author={R{\"o}dl, V.},
      author={Tengan, E.},
      author={Schacht, M.},
      author={Tokushige, N.},
       title={Density theorems and extremal hypergraph problems},
        date={2006},
     journal={Israel Journal of Mathematics},
      volume={152},
      number={1},
       pages={371\ndash 380},
}

\bib{SW17}{article}{
      author={Saad, Alex},
      author={Wolf, Julia},
       title={Ramsey multiplicity of linear patterns in certain finite abelian
  groups},
        date={2017},
     journal={The Quarterly Journal of Mathematics},
      volume={68},
      number={1},
       pages={125\ndash 140},
}

\bib{S10}{article}{
      author={Sanders, Tom},
       title={Green's sumset problem at density one half},
    language={eng},
        date={2011},
     journal={Acta Arithmetica},
      volume={146},
      number={1},
       pages={91\ndash 101},
         url={http://eudml.org/doc/279447},
}

\bib{S88}{article}{
      author={Shelah, Saharon},
       title={{Primitive recursive bounds for van der Waerden numbers}},
        date={1988},
     journal={Journal of the American Mathematical Society},
      volume={1},
      number={3},
       pages={683\ndash 697},
}

\bib{S91}{article}{
      author={Sidorenko, Alexander~F.},
       title={Inequalities for functionals generated by bipartite graphs},
        date={1991},
     journal={Diskretnaya Matematika},
      volume={3},
      number={3},
       pages={50\ndash 65},
}

\bib{Sp79}{article}{
      author={Spencer, Joel~H.},
       title={Ramsey's theorem for spaces},
        date={1979},
        ISSN={0002-9947},
     journal={Trans. Amer. Math. Soc.},
      volume={249},
      number={2},
       pages={363\ndash 371},
         url={https://doi.org/10.2307/1998796},
      review={\MR{525678}},
}

\bib{T81}{article}{
      author={Taylor, Alan~D.},
       title={Bounds for the disjoint unions theorem},
        date={1981},
        ISSN={0097-3165},
     journal={Journal of Combinatorial Theory, Series A},
      volume={30},
      number={3},
       pages={339\ndash 344},
  url={https://www.sciencedirect.com/science/article/pii/0097316581900315},
}

\bib{T22}{article}{
      author={Tyrrell, Fred},
       title={New lower bounds for cap sets},
        date={2022},
     journal={arXiv:2209.10045},
}

\bib{V21}{article}{
      author={Versteegen, Leo},
       title={{Common and Sidorenko equations in Abelian groups}},
        date={2021},
     journal={arXiv:2109.04445},
}

\bib{V23}{article}{
      author={Versteegen, Leo},
       title={Linear configurations containing 4-term arithmetic progressions
  are uncommon},
        date={2023},
     journal={Journal of Combinatorial Theory, Series A},
      volume={200},
       pages={105792},
}

\end{biblist}
\end{bibdiv}

\end{document}